\documentclass[10pt]{amsart}
\usepackage{amssymb,amstext,amsmath,amscd,amsthm,amsfonts,enumerate,graphicx,latexsym,color,stmaryrd}
\usepackage[all]{xy}
\usepackage{comment}
\usepackage{cite}

\newtheorem{thm}{Theorem}[section]
\newtheorem{cor}[thm]{Corollary}
\newtheorem{lem}[thm]{Lemma}
\newtheorem{prop}[thm]{Proposition}
\theoremstyle{definition}
\newtheorem{dfn}[thm]{Definition}
\newtheorem{ques}[thm]{Question}
\newtheorem{rem}[thm]{Remark}

\theoremstyle{remark}

\newtheorem*{claim*}{Claim}

\numberwithin{equation}{thm}

\def\db{\operatorname{D^b}}
\def\dm{\operatorname{D^{\mbox{\boldmath$-$}}}}
\def\ds{\operatorname{D_{sg}}}
\def\dpf{\operatorname{D^{perf}}}
\def\thick{\operatorname{thick}}
\def\K{\operatorname{K}}
\def\H{\operatorname{H}}

\def\pd{\operatorname{pd}}
\def\cx{\operatorname{cx}}
\def\height{\operatorname{ht}}
\def\codim{\operatorname{codim}}

\def\Ext{\operatorname{\mathsf{Ext}}}

\def\spc{\operatorname{\mathsf{Esp}}}
\def\bspc{\operatorname{\mathsf{Spc}}}
\def\spec{\operatorname{\mathsf{Spec}}}
\def\sing{\operatorname{\mathsf{Sing}}}
\def\rad{\operatorname{\mathsf{Rad}}}
\def\supp{\operatorname{\mathsf{Supp}}}
\def\spp{\operatorname{\mathsf{Spp}}}
\def\bspp{\operatorname{\mathsf{supp}}}
\def\ssupp{\operatorname{\mathsf{\underline{Supp}}}}
\def\st{\operatorname{\mathsf{Param}}}

\def\V{\operatorname{\mathsf{V}}}
\def\Z{\operatorname{\mathsf{Z}}}
\def\th{\operatorname{\mathsf{Th}}}
\def\spcl{\operatorname{\mathsf{Spcl}}}
\def\thom{\operatorname{\mathsf{Thom}}}
\def\Tame{\operatorname{\mathsf{Tame}}}

\def\m{\mathfrak{m}}
\def\p{\mathfrak{p}}
\def\q{\mathfrak{q}}
\def\r{\mathfrak{r}}
\def\pp{\mathfrak{s}}

\def\P{\mathcal{P}}
\def\Q{\mathcal{Q}}

\def\C{\mathcal{C}}
\def\D{\mathcal{D}}
\def\E{\mathcal{E}}
\def\s{\mathcal{S}}
\def\T{\mathcal{T}}
\def\X{\mathcal{X}}
\def\Y{\mathcal{Y}}
\def\cZ{\mathcal{Z}}

\def\II{\mathbb{I}}
\def\PP{\mathbb{P}}
\def\RR{\mathbb{R}}
\def\ZZ{\mathbb{Z}}
\def\XX{\mathbb{X}}

\def\tame{\mathsf{tame}}
\def\inc{\mathrm{inc}}
\def\zero{\mathbf{0}}
\def\xx{\boldsymbol{x}}
\title[Construction of spectra of triangulated categories without tensor structure]{Construction of spectra of triangulated categories without tensor structure and applications to commutative rings}
\author{Hiroki Matsui}
\address{Graduate School of Mathematics, Nagoya University, Furocho, Chikusaku, Nagoya, Aichi 464-8602, Japan}
\email{m14037f@math.nagoya-u.ac.jp}
\author{Ryo Takahashi}
\address{Graduate School of Mathematics, Nagoya University, Furocho, Chikusaku, Nagoya, Aichi 464-8602, Japan/Department of Mathematics, University of Kansas, Lawrence, KS 66045-7523, USA}
\email{takahashi@math.nagoya-u.ac.jp}
\urladdr{http://www.math.nagoya-u.ac.jp/~takahashi/}
\subjclass[2010]{13D09, 13H10, 18E30}
\keywords{complete intersection, derived category, hypersurface, perfect complex, singularity category, spectrum, (prime/radical/tame) thick subcategory, triangulated category}
\thanks{HM was partly supported by JSPS Grant-in-Aid for JSPS Fellows 16J01067. RT was partly supported by JSPS Grants-in-Aid for Scientific Research 16K05098 and JSPS Fund for the Promotion of Joint International Research 16KK0099}
\begin{document}
\begin{abstract}
In this paper, as an analogue of the spectrum of a tensor triangulated category introduced by Balmer, we define a spectrum of a triangulated category which does not necessarily admit a tensor structure.
We apply it for some triangulated categories associated to a commutative noetherian ring.
\end{abstract}
\maketitle
\section{Introduction}

Classification of thick subcategories has been one of the main approaches in the studies of triangulated categories for several decades.
It has been studied so far in many areas such as stable homotopy theory, modular representation theory, algebraic geometry, commutative/non-commutative algebra and so on; see \cite{BCR,BIK,BIKP,FP,Hop,HS,Nee,Tak10,Ste,Th} and references therein.
In commutative algebra, Hopkins \cite{Hop} and Neeman \cite{Nee} classified thick subcategories of derived categories of perfect complexes over commutative rings via Zariski spectra.
The second author \cite{Tak10} classified thick subcategories of singularity categories of hypersurfaces via singular loci, and it was extended to complete intersections by Stevenson \cite{Ste}.

On the other hand, there is a beautiful theory initiated by Balmer \cite{Bal05} which is called {\em tensor triangular geometry}.
He introduced the concepts of thick tensor ideals, radical thick tensor ideals and prime thick tensor ideals of an essentially small tensor triangulated category $\T$ as analogues of ideals, radical ideals and prime ideals of commutative rings.
Then he defined a topology on the set $\bspc \T$ of prime thick tensor ideals of $\T$, which is called the {\it Balmer spectrum} of $\T$.
He accomplished the following monumental work in this theory, which enables us to do algebro-geometric studies of tensor triangulated categories.

\begin{thm}[Balmer]\label{0Balcls}
Let $\T$ be an essentially small tensor triangulated category.
Taking the Balmer supports gives a bijection between the radical thick tensor ideals of $\T$ and the Thomason subsets of $\bspc \T$.
\end{thm}

The Balmer theory works for arbitrary (essentially small) tensor triangulated categories, but by definition, it does not (at least directly) work for a triangulated category without a tensor structure.
That is why in this paper we make a variant of Balmer's theory for a general triangulated category $\T$.
Once we assign a class $\spc \T$ of thick subcategories which we call {\it prime thick subcategories}, we can define a topology on $\spc\T$ and prove the following analogous result to Theorem \ref{0Balcls}.

\begin{thm}[Theorem \ref{class}]
Let $\T$ be an essentially small triangulated category.
Then there exists a one-to-one correspondence
$$
\rad\T\overset{\text{\rm{1-1}}}{\longleftrightarrow}\st\T.
$$
\end{thm}

\noindent
Let us explain the notation used above:
$\rad\T$ is the set of {\em radical thick subcategories} of $\T$, which are variants of radical thick tensor ideals of a tensor triangulated category, while $\st\T$ consists of certain subsets of $\spc\T$, which parametrizes the radical thick subcategories as stated above.

Thus, an important and essential point is to find out a suitable class of prime thick subcategories for a given triangulated category.
In this paper, we do this for the bounded derived category $\db(R)$ of finitely generated $R$-modules, the derived category $\dpf(R)$ of perfect $R$-complexes, and the singularity category $\ds(R):=\db(R)/\dpf(R)$, where $R$ is a commutative noetherian ring.
For such triangulated categories $\T$, we introduce a pair of maps
$$
\xymatrix{
\spc\T\ \ar@<.5mm>[rr]^\pp&&\ \XX_\T,\ar@<.5mm>[ll]^\s
}
$$ 
where we set $\XX_{\db(R)}=\XX_{\dpf(R)}=\spec R$ and $\XX_{\ds(R)} = \sing R$.
We define {\em tame thick subcategories} of $\T$ to be thick subcategories representable by the corresponding support.
A tame thick subcategory is always radical, and it is natural to consider when the converse holds.
Our results include the following.

\begin{thm}[Theorems \ref{bij}, \ref{cmpr}, \ref{dim}, \ref{sgbij}, \ref{10}, \ref{dimsg} and Corollary \ref{pppp}]\
\begin{enumerate}[\rm(1)]
\item
Let $R$ be a commutative noetherian ring, and $\T=\dpf(R)$.
Then the following hold true.
\begin{enumerate}[\rm(i)]
\item
The maps $\pp$ and $\s$ are mutually inverse homeomorphisms.
\item
The tame thick subcategories coincide with the radical ones.
\end{enumerate}
Furthermore, $\st\T$ coincides with the set of Thomason subsets of $\spc\T$.
\item
Let $R$ be a local complete intersection.
Suppose either of the following.
\begin{itemize}
\item
$\T=\db(R)$, $\dim R>0$ and $\PP={\rm regular}$.
\item
$\T=\ds(R)$, $R$ is excellent and $\PP={\rm hypersurface}$.
\end{itemize}
Then the following are equivalent.
\begin{enumerate}[\rm(i)]
\item
The topological spaces $\spc \T$ and $\XX_\T$ have the same Krull dimension.
\item
The topological spaces $\spc \T$ and $\XX_\T$ are homeomorphic.
\item
The maps $\pp$ and $\s$ are mutually inverse homeomorphisms.
\item
The tame thick subcategories coincide with the radical ones.
\item
The ring $R$ is $\PP$.
\end{enumerate}
When these equivalent conditions hold, $\st \T$ coincides with the set of Thomason subsets of $\spc \T$.
\end{enumerate}
\end{thm}

We should remark that recently the first author \cite{Mat} has also defined a spectrum of a triangulated category, which is different from a spectrum introduced in the present paper.
More precisely, the spectrum introduced in \cite{Mat} is defined as the set of {\it irreducible subsets} of $\T$ together with a topology, which is uniquely determined by using the triangulated structure of $\T$.
The advantage of a spectrum introduced in the present paper is that there is flexibility in the choice of the underlying set; for a given triangulated category, choosing appropriate prime thick subcategories makes it more manageable.

This paper is organized as follows.
In Section 2, following Balmer's work \cite{Bal05}, for each triangulated category $\T$ we assign a class $\spc \T$ of prime thick subcategories, and define a topology on it. Several analogous notions to Balmer's theory, such as supports and radical thick subcategories, are also introduced.
Later sections are devoted to application to commutative noetherian rings $R$.
The spectra of the derived categories $\db(R),\dpf(R)$ and intermediate triangulated categories are explored in Section 3, and the spectrum of the singularity category $\ds(R)$ is investigated in Section 4 as well.

\section{Spectra of triangulated categories}

In this section, following Balmer's theory \cite{Bal05}, we define a topological space for a triangulated category which does not necessarily possess a tensor structure, and by using it we give a classification of a certain class of thick subcategories.

Throughout this section, let $\T$ be an essentially small triangulated category.
Denote by $\th\T$ the set of thick subcategories of $\T$.
Fix a subset $\spc \T$ of $\th\T$.
We call elements of $\spc \T$ {\it prime thick subcategories} of $\T$.
First we introduce a topology on the set $\spc \T$.

\begin{dfn}[{cf. \cite[Definition 2.1]{Bal05}}]\label{top}
For a class $\E$ of objects of $\T$, we put
$$
\Z(\E):= \{ \P \in \spc \T \mid \P \cap \E = \emptyset \}.
$$
One can easily check that the following statements hold; note that the fact that $\T$ is closed under direct summands enables the last one to hold.
\begin{itemize}
\item 
$\Z(\T) = \emptyset$ and $\Z(\emptyset) = \spc \T$.
\item 
$\bigcap_{i \in I}\Z(\E_i) = \Z(\bigcup_{i \in I} \E_i)$ for a family $\{\E_i\}_{i\in I}$ of classes of objects of $\T$.
\item 
$\Z(\E) \cup \Z(\E') = \Z(\E \oplus \E')$ for classes $\E,\E'$ of objects of $\T$, where $\E\oplus\E':=\{X\oplus X'\mid X\in\E,\,X'\in\E'\}$.
\end{itemize}
Thus $\spc\T$ is a topological space with the closed subsets $\Z(\E)$.
We then call $\spc \T$ the {\it spectrum} of $\T$.\footnote{For a tensor triangulated category $\T$, the spectrum of $\T$ in our sense does not necessarily coincide with the Balmer spectrum $\bspc\T$ of $\T$. To avoid confusion, instead of $\bspc\T$ we adopt the notation $\spc\T$, which comes from the Spanish translation {\em espectro} of spectrum.}

For an object $M \in \T$, we define its {\it support} by
$$
\spp M := \Z(\{M \}) = \{\P \in \spc \T \mid M \not\in \P \}.
$$
It directly follows by definition that the equality $\Z(\E) = \bigcap_{M \in \E} \spp M$ holds for each $\E$.
This shows that the family $\{\spp M\}_{M \in \T}$ of closed subsets forms a closed basis of $\spc \T$.
For a thick subcategory $\X$ of $\T$, we define its {\it support} by
$$
\spp \X := \bigcup_{M \in \X} \spp M=\{\P\in\spc\T\mid\X\nsubseteq\P\},
$$
which is a specialization-closed subset of $\spc \T$.
\end{dfn}

For a full subcategory $\E$ of $\T$, we denote by $\thick \E$ the smallest thick subcategory of $\T$ containing $\E$.
We provide some basic properties of supports, which says that the pair $(\spc \T, \spp)$ is a support data for $\T$ in the sense of \cite{Mat}.

\begin{rem}\label{11}
\begin{enumerate}[\rm(1)]
\item 
$\spp(0) = \emptyset$.
\item 
$\spp(M[n]) = \spp M$ for $M \in \T$ and $n \in \mathbb{Z}$.
\item 
$\spp(M \oplus N) = \spp M \cup \spp N$ for $M, N \in \T$.
\item 
$\spp M\subseteq \spp L\cup \spp N$ for an exact triangle $L \to M \to N \to L[1]$ in $\T$.
\end{enumerate}
\noindent
In particular, for a subset $S$ of $\spc\T$, the full subcategory
$$
\spp^{-1}S:=\{M\in\T\mid\spp M\subseteq S\}
$$
of $\T$ is a thick subcategory, which implies that $\spp\X=\spp(\thick\X)$ for each full subcategory $\X$ of $\T$.
\end{rem}

\begin{proof}
This is a direct consequence of the definition of a thick subcategory.
Indeed, the fact that a thick subcategory contains the zero object shows (1), the fact that it is closed under shifts shows (2), the fact that it is closed under finite direct sums and direct summands shows (3), and the fact that it satisfies the 2-out-of-3 property shows (4).
\end{proof}

The following proposition is proved in \cite{Bal05} for Balmer spectra of tensor triangulated categories.
We can prove the same statement for our topological space $\spc \T$ and the proof is completely the same.

\begin{prop}[{cf. \cite[Proposition 2.9]{Bal05}}] \label{t0}
For any prime thick subcategory $\P$ of $\T$, one has
$$
\overline{\{\P\}} = \{\Q \in \spc \T \mid \Q \subseteq \P \}.
$$
In particular, $\spc \T$ is a $T_0$-space.
\end{prop}

\begin{proof}
Let $S$ be the right-hand side.
Then $\P\in S=\Z(\P^\complement)$.
If $\Z(\E)$ contains $\P$, then $\P\cap\E=\emptyset$, and $\Q\cap\E=\emptyset$ for all $\Q\in S$, which shows $S\subseteq\Z(\E)$.
Now the assertion follows.
\end{proof}

Next, we define the height of a prime thick subcategory and the dimension of $\T$.
 
\begin{dfn}
For a prime thick subcategory $\P$ of $\T$, we define the {\it height} $\height \P$ of $\P$ to be the largest number $n$ such that there exists a chain $\P_0 \subsetneq \P_1 \subsetneq \cdots \subsetneq \P_n =\P$ of prime thick subcategories of $\T$.
We define the {\it dimension} $\dim \T$ of $\T$ as the supremum of the heights of prime thick subcategories of $\T$.
\end{dfn}

Let $X$ be a topological space.
Recall that the {\it Krull dimension} $\dim X$ of $X$ is defined as the supremum of integers $n\ge0$ such that there exists a chain $Z_0 \subsetneq Z_1 \subsetneq \cdots \subsetneq Z_n$ of irreducible closed subsets of $X$.
Recall also that $X$ is {\em sober} if every irreducible closed subset of $X$ is the closure of exactly one point of $X$.
A typical example of a sober topological space is the Zariski spectrum $\spec R$ of a commutative ring $R$.

\begin{rem}\label{r}
By Proposition \ref{t0}, for two prime thick subcategories $\P, \Q$ of $\T$ one has $\overline{\{\P\}} \subsetneq \overline{\{\Q\}}$ if and only if $\P \subsetneq \Q$.
This shows $\dim\T\le\dim\spc \T$.
Moreover, if $\spc \T$ is sober, then $\dim \spc \T = \dim \T$. 
\end{rem}

Using the spectrum of $\T$, we can classify a certain class of thick subcategories characterized by $\spc \T$, which we call radical thick subcategories.

\begin{dfn}[{cf. \cite[Lemma 4.2]{Bal05}}]\label{radef}
For a thick subcategory $\X$ of $\T$, define the {\it radical} of $\X$ by 
$$
\sqrt{\X} := \bigcap_{\P \in \spc \T,\,\X \subseteq \P} \P.
$$
We say that a thick subcategory $\X$ is {\it radical} if $\sqrt{\X} = \X$.
Denote by $\rad \T$ the set of radical thick subcategories of $\T$.
\end{dfn}

The following proposition tells us that the support cannot distinguish a thick subcategory and its radical.
Thus to classify thick subcategories by their supports we focus on the radical thick subcategories.

\begin{prop}\label{radsp}
\begin{enumerate}[\rm(1)]
\item
For $M\in\T$ one has $\spp M=\emptyset$ if and only if $M\in\sqrt{0}$.
\item
For each $\X \in \th\T$ there is an equality $\spp \sqrt{\X} = \spp \X$.
\end{enumerate}
\end{prop}

\begin{proof}
(1) This is straightforward from the equality $\sqrt{0}=\bigcap_{\P \in \spc \T}\P$.

(2) Fix any $\P \in \spc \T$.
One then has $\P \in \spp \X$ if and only if $\X \not\subseteq \P$, if and only if $\sqrt{\X} \not\subseteq \P$, if and only if $\P \in \spp \sqrt{\X}$.
\end{proof}

We define the {\em parameter set} $\st\T$ as the set of supports of thick subcategories of $\T$:
$$
\st\T:=\{\spp \X \mid \X \in \th\T \}\subseteq2^{\spc\T}.
$$
The reason why we call this so is that it parametrizes the radical thick subcategories of $\T$ as follows.

\begin{thm}\label{class}
There is a one-to-one correspondence
$$
\xymatrix{
\rad\T \ar@<0.5ex>[r]^-{\spp} &
\st\T\ar@<0.5ex>[l]^-{\spp^{-1}}. 
}
$$
\end{thm}

\begin{proof}
Let $\X\in\th\T$ and $M\in\T$.
The condition $\spp M \subseteq \spp \X$ is equivalent to saying that for each $\P \in \spc \T$, if $\P$ contains $\X$, then $\P$ also contains $M$.
Therefore we have equalities $\spp^{-1}(\spp \X) = \bigcap_{\X \subseteq \P} \P = \sqrt{\X}$.
This shows that $\spp^{-1} : \st\T \to \rad\T$ is a well-defined map, and that it is a retraction of the map $\spp : \rad\T \to \st\T$.
Applying Proposition \ref{radsp}(2), we get equalities $\spp (\spp^{-1} (\spp \X)) = \spp \sqrt{\X} = \spp \X$, which completes the proof of the theorem.
\end{proof}

A subset of a topological space $X$ is said to be {\it Thomason} if it is the union of some closed subsets of $X$ whose complements are quasi-compact.
We denote by $\thom X$ the set of Thomason subsets of $X$.
Theorems \ref{class} and \ref{0Balcls} naturally lead us to the following question.

\begin{ques}\label{5}
Can we characterize the parameter set $\st\T$ in terms of the topology of $\spc\T$?
More specifically, does the equality $\st\T= \thom(\spc \T)$ hold?
\end{ques}
\noindent
We will give partial answers to this question in Sections 3 and 4.

\section{Spectra of derived categories}
In this section, we consider what we got in the previous section for derived categories of commutative noetherian rings, that is, we discuss those spectra.

Throughout this section, let $R$ be a commutative noetherian ring, and let $\D$ be a triangulated subcategory of the bounded derived category $\db(R)$ of finitely generated $R$-modules containing the derived category $\dpf(R)$ of perfect $R$-complexes.
The purpose of this section is to investigate $\spc\D$.
To begin with, we recall the notion of a homological support and its basic properties for later use.

\begin{dfn}
For each object $M \in \D$, the {\it homological support} of $M$ is defined by
$$
\supp M := \{\p \in \spec R \mid\H(M)_\p \ne 0 \}=\{\p \in \spec R \mid M_\p \not\cong 0 \}.
$$
For a full subcategory $\X \subseteq \D$, the {\it homological support} of $\X$ is defined by
$$
\supp\X := \bigcup_{M \in \X} \supp M.
$$
\end{dfn}

For an ideal $I$ of $R$, let $\K(I)$ stand for the Koszul complex of a system of generators of $I$.
We can easily verify that the homological support satisfies the following properties.

\begin{rem}\label{2}
\begin{enumerate}[\rm(1)]
\item
For an ideal $I$ of $R$ it holds that $\supp\K(I)=\V(I)$.
\item
For $M\in\D$ the set $\supp M$ is Zariski-closed (i.e. closed in $\spec R$).
\item 
For $M\in\D$ one has $\supp M = \emptyset$ if and only if $M \cong 0$ in $\D$.
\item
For $M\in\D$ one has $\supp(M[1])=\supp M$.
\item 
For $M,N\in\D$ there is an equality $\supp(M \oplus N) = \supp M \cup \supp N$.
\item 
For an exact triangle $L \to M \to N \to L[1]$ in $\D$ there is an inclusion $\supp M\subseteq \supp L\cup \supp N$.
\end{enumerate}
\noindent
In particular, for a subset $W$ of $\spec R$, the full subcategory
$$
\supp^{-1}_{\D}(W) := \{M \in \D \mid \supp M \subseteq W \}
$$
of $\D$ is a thick subcategory of $\D$.
\end{rem}

For a partially ordered set $S$, let $\max S$ (resp. $\min S$) be the set of maximal (resp. minimal) elements of $S$.
We always regard a set of sets as a partially ordered set with respect to the inclusion relation.

To define the spectrum of $\D$, we have to fix a class of prime thick subcategories.
We adopt the following definition of prime thick subcategories of $\D$.

\begin{dfn}\label{7}
For $\X \in \th\D$, we define the set $\II(\X)$ of ideals of $R$ by
$$
\II(\X):=\{I \subseteq R \mid \K(I) \not\in \X\}.
$$
This is well-defined since the condition $\K(I) \not\in \X$ is independent of the choice of a system of generators of $I$; see \cite[Proposition 1.6.21]{BH}.
We say that a thick subcategory $\P$ of $\D$ is {\it prime} if $\II(\P)$ has a unique maximal element, i.e., $\#\max\II(\P)=1$.
When this is the case, we denote by $\pp(\P)$ the maximal element of $\II(\P)$.
Let $\spc \D$ denote the set of prime thick subcategories of $\D$.
We equip $\spc\D$ with the topology defined in Definition \ref{top}.
\end{dfn}

\begin{rem}
The motivation of this definition of a prime thick subcategory of $\D$ comes from \cite[Proposition 3.7]{MT}, which states that $\II(\P)$ has a unique maximal element for a prime thick tensor ideal $\P$ of the right bounded derived category $\dm(R)$ of finitely generated $R$-modules.
\end{rem}

Next, we give a typical example of prime thick subcategories.
For each prime ideal $\p$ of $R$, we define the full subcategory $\s(\p)$ of $\D$ by
$$
\s(\p):= \{M \in \D \mid M_\p \cong 0\}.
$$

\begin{lem}\label{sprm2}
Let $\p\in\spec R$.
Then $\s(\p)$ is a prime thick subcategory of $\D$ with $\pp(\s(\p)) = \p$.
\end{lem}

\begin{proof}
It is easy to check that
\begin{equation}\label{1}
\s(\p)= \supp_{\D}^{-1} \{\q \in \spec R \mid \q \nsubseteq \p \},
\end{equation}
which especially says that $\s(\p)$ is a thick subcategory of $\D$.
For an ideal $I$ of $R$, one has $\K(I)\not\in\s(\p)$ if and only if $\K(I)_\p \not\cong 0$, if and only if $\p \in \supp \K(I) = \V(I)$.
Hence $\II(\s(\p))= \{I \subseteq R \mid \p \in \V(I)\}$, and we see that $\max\II(\s(\p))=\{\p\}$.
\end{proof}

Lemma \ref{sprm2} says that $\s$ assigns to each prime ideal of $R$ a prime thick subcategory of $\D$.
Conversely, $\pp$ also assigns to each prime thick subcategory of $\D$ a prime ideal of $R$. 

\begin{lem}\label{prm}
For a thick subcategory $\X$ of $\D$ every maximal element of $\II(\X)$ is a prime ideal of $R$.
In particular, $\pp(\P)$ is a prime ideal of $R$ for any prime thick subcategory $\P$ of $\D$.
\end{lem}

\begin{proof}
Take a maximal element $\p$ of $\II(\X)$.
Note then that $\p\ne R$.
Consider elements $a, b \in R$ with $ab \in \p$.
By the octahedral axiom, there is an exact triangle $\K(a) \to \K(ab) \to \K(b) \to \K(a)[1]$ in $\D$.
Tensoring $\K(\p)$ with this, we obtain an exact triangle
$
\K(\p+(a)) \to \K(\p+(ab)) \to \K(\p+(b)) \to \K(\p+(a))[1]
$
in $\D$.
Since $\p+(ab) = \p$, the complex $\K(\p+(ab))$ does not belong to $\X$.
Therefore, either $\K(\p+(a))$ or $\K(\p+(b))$ is outside $\X$.
The maximality of $\p$ implies that $a \in \p$ or $b \in \p$.
\end{proof}

Lemmas \ref{sprm2} and \ref{prm} lead us to the construction of a pair of maps of topological spaces
$$
\pp: \spc \D \rightleftarrows \spec R :\s.
$$
These maps are basic tools to investigate the structure of $\spc \D$ by comparison with $\spec R$.

It is hard to check from the definition whether a given thick subcategory is prime or not.
We give a useful characterization of prime thick subcategories in terms of their homological supports.

\begin{thm}\label{prmsup}
Assume either
{\rm(a)} $\D = \dpf(R)$, or
{\rm(b)} $R$ is a local complete intersection.
Then for a thick subcategory $\P$ of $\D$ it holds that
$$
\text{$\P$ is prime}
\quad\Longleftrightarrow\quad
\text{$\supp \P = \{ \q \in \spec R \mid \q \not\subseteq \p \}$ for some $\p \in \spec R$.}
$$
When the latter condition holds, one has $\p = \pp(\P)$.
\end{thm}

\begin{proof}
We prove the theorem step by step.

(1) Let $W$ be a specialization-closed subset of $\spec R$, and let $\p$ be a prime ideal of $R$.
Then it is straightforward that $\max(W^\complement)=\{\p\}$ if and only if $W= \{ \q \in \spec R \mid \q \not\subseteq \p \}$.

(2) Let $\X$ be a thick subcategory of $\D$.
Then $\max(\II(\X)\cap\spec R)=\max\II(\X)$.
In fact, by Lemma \ref{prm} implies there are inclusions $\max\II(\X)\subseteq\II(\X)\cap\spec R\subseteq\II(\X)$.
It remains to note the general fact that for two partially ordered sets $A,B$ with $\max B\subseteq A\subseteq B$ one has $\max A=\max B$.

(3) For each thick subcategory $\X$ of $\D$, the equality $\II(\X) = \{I \subseteq R \mid \V(I) \not\subseteq \supp \X\}$ holds.
In fact, case (a) follows from \cite[Theorem 1.5]{Nee}.
Case (b) will follow if we show that $\K(I) \in \X$ if and only if $\V(I) \subseteq \supp \X$ for an ideal $I$ of $R$ and a thick subcategory $\X$ of $\D$.
Taking the homological supports shows the `only if' part.
For the `if' part, as there exist only finitely many minimal primes of $I$, we can find a complex $M\in\X$ such that $\V(I)\subseteq\supp M$.
Since $R$ is assumed to be a local complete intersection, it follows from \cite[Theorem 5.2]{Pol} that $M$ is {\em proxy small}, and by \cite[Proposition 4.4]{DGI} the Koszul complex $\K(I)$ is in $\thick M$, and hence it belongs to $\X$.

(4) For $\X\in\th\D$ one has $\II(\X)\cap\spec R=(\supp\X)^\complement$.
Indeed, for a prime ideal $\p$ of $R$ it holds that $\V(\p)\nsubseteq\supp\X$ if and only if $\p\in(\supp\X)^\complement$.
Using (3), we deduce the equality.

(5) Now let us show the assertion of the theorem.
Let $\P$ be a thick subcategory of $\D$, and let $\p$ be a prime ideal of $R$.
Then $\supp\P=\{\q\in\spec R\mid\q\nsubseteq\p\}$ if and only if $\max((\supp\P)^\complement)=\{\p\}$ by (1), if and only if $\max(\II(\P)\cap\spec R)=\{\p\}$ by (4), if and only if $\max\II(\P)=\{\p\}$ by (2), if and only if the thick subcategory $\P$ is prime with $\pp(\P)=\{\p\}$.
\end{proof}

Here we consider composition and continuity of the maps $\pp$ and $\s$.

\begin{cor}\label{maps2}
\begin{enumerate}[\rm(1)]
\item 
The maps $\s$ and $\pp$ are order-reversing with $\pp \cdot \s = 1$.
Moreover, $\s$ is continuous.
\item
Assume either that $\D = \dpf(R)$, or that $R$ is a local complete intersection.
Then $\pp$ is also continuous, and $\P \subseteq \s(\pp(\P))) = \supp^{-1}_{\D}(\supp \P)$ for any prime thick subcategory $\P$ of $\D$.
\end{enumerate}
\end{cor}

\begin{proof}
(1) It is directly verified that $\s,\pp$ are order-reversing.
Lemma \ref{sprm2} shows $\pp\cdot\s=1$.
The equality
\begin{equation}\label{3}
\s^{-1}(\spp M) =\supp M
\end{equation}
holds for each $M\in\D$, which shows that $\s$ is a continuous map.

(2) Fix an ideal $I$ of $R$ and a prime thick subcategory $\P$ of $\D$.
If $\K(I)$ is not in $\P$, then $I$ belongs to $\II(\P)$ and is contained in $\pp(\P)$.
Conversely, if $\K(I)$ is in $\P$, then $\V(I)$ is contained in $\supp\P$.
The latter set consists of the prime ideals that are not contained in $\pp(\P)$ by Theorem \ref{prmsup}.
In particular, $\pp(\P)$ does not contain $I$.
This shows that $\pp^{-1}(\V(I)) = \spp \K(I)$, whence $\pp$ is continuous.
The equality $\s(\pp(\P))) = \supp^{-1}_{\D}(\supp \P)$ follows by Theorem \ref{prmsup} and \eqref{1}.
\end{proof}

We may ask when the maps $\pp$ and $\s$ are mutually inverse bijections.
Here is an answer to this question.

\begin{thm}\label{bij}
\begin{enumerate}[\rm(1)]
\item
One has mutually inverse homeomorphisms
$$
\xymatrix{
\spc\dpf(R)\ar@<0.5ex>[r]^-\pp & \spec R\ar@<0.5ex>[l]^-\s.
}
$$
\item
Let $R$ be a local complete intersection.
Then the following mutually inverse bijections are induced.
$$
\xymatrix{
\max(\spc \D) \ar@<0.5ex>[r]^-\pp & \min(\spec R)\ar@<0.5ex>[l]^-\s.
}
$$
\item
Let $R$ be a local complete intersection of positive Krull dimension.
Then the following are equivalent.
\begin{enumerate}[\rm(i)]
\item 
The maps $\pp: \spc \D \rightleftarrows \spec R :\s$ are mutually inverse bijections.
\item 
The maps $\pp: \spc \D \rightleftarrows \spec R :\s$ are mutually inverse homeomorphisms.
\item
One has $\D = \dpf(R)$.
\end{enumerate}
\end{enumerate}
\end{thm}

\begin{proof}
(1) We use Corollary \ref{maps2}.
It suffices to show that $\s$ is surjective.
Let $\P\in\spc\dpf(R)$.
Then we have $\s(\pp(\P))=\supp^{-1}_{\dpf(R)}(\supp\P)=\P$, whose last equality follows from \cite[Theorem 1.5]{Nee}.

(2) Again, we use Corollary \ref{maps2}.
We have only to show that $\s : \min(\spec R) \to \max  (\spc\D)$ is a well-defined surjection.
First, let $\p\in\min(\spec R)$.
Take a prime thick subcategory $\P$ of $\D$ containing $\s(\p)$.
We have $\pp(\P) \subseteq \pp(\s(\p)) = \p$, and $\p = \pp(\P)$ by the minimality of $\p$.
Therefore $\P \subseteq \s(\pp(\P))=\s(\p)$, and we get $\P=\s(\p)$.
Thus the map $\s : \min(\spec R) \to \max  (\spc\D)$ is well-defined.
Next, take $\P\in\max(\spc\D)$.
The inclusion $\P \subseteq \s(\pp(\P))$ and the maximality of $\P$ show that $\P = \s(\pp(\P))$.
Choose a minimal prime $\p$ of $R$ contained in $\pp(\P)$.
Then $\P =\s(\pp(\P)) \subseteq \s(\p)$.
Again the maximality of $\P$ implies $\P = \s(\p)$.
Thus $\s : \min(\spec R) \to \max  (\spc\D)$ is surjective.

(3) It follows from the assertion (1) that (iii) implies (ii), while it is trivial that (ii) implies (i).
Let us show that (i) implies (iii).
Denote by $\m$ the maximal ideal of $R$.
Since $\dim R > 0$, there is a prime ideal $\p$ of $R$ strictly contained in $\m$.
Set $\P = \thick\{\K(\q) \mid\q\in\spec R,\,\q \nsubseteq \p\}$.
Then we can verify that the equality $\supp \P = \{\q \in \spec R \mid \q \not\subseteq \p\}$ holds.
Theorem \ref{prmsup} implies that $\P$ is a prime thick subcategory with $\pp(\P)=\p$, whence $\P =\s(\pp(\P))= \s(\p)$.
We have $\supp_{\D}^{-1}\{\m\} \subseteq\s(\p) = \P \subseteq \dpf(R)$.

Take any $X\in\D$.
Taking a truncation of a projective resolution of $X$ gives rise to an exact triangle $P \to X \to M[n] \rightsquigarrow$ in $\db(R)$, where $P$ is a perfect complex, $M$ is a maximal Cohen--Macaulay $R$-module and $n$ is an integer.
Let $\xx$ be a maximal regular sequence on $R$.
Then $\xx$ is also a regular sequence on $M$, and $M/\xx M\in\supp_{\D}^{-1}\{\m\} \subseteq \dpf(R)$.
This shows that $M/\xx M$ has finite projective dimension as an $R$-module.
It follows from \cite[Exercise 1.3.6]{BH} that $M\in\dpf(R)$, and we conclude that $X \in \dpf(R)$.
\end{proof}

Applying Theorem \ref{bij}(3) to $\D = \db(R)$, we immediately obtain the following result.

\begin{cor}\label{pppp}
Let $R$ be a local complete intersection with $\dim R>0$.
Then the following are equivalent.
\begin{enumerate}[\rm(1)]
\item 
The maps $\pp: \spc \db(R) \rightleftarrows \spec R :\s$ are mutually inverse bijections.
\item 
The maps $\pp: \spc \db(R) \rightleftarrows \spec R :\s$ are mutually inverse homeomorphisms.
\item
The ring $R$ is regular.
\end{enumerate}
\end{cor}

\begin{rem}\label{art}
If $R$ is an artinian complete intersection local ring, then the maps $\pp: \spc \db(R) \rightleftarrows \spec R :\s$ are always mutually inverse homeomorphism.
Indeed, write $\spec R=\{\m\}$.
Let $\P$ be a prime thick subcategory of $\db(R)$.
Then Theorem \ref{prmsup} implies that $\supp \P = \{\p \in \spec R \mid \p \not\subseteq \m\} = \emptyset$.
This implies $\P = \zero$, and thus $\spc\db(R)=\{0\}$.
It remains to note that $\s(\m)=\zero$ and $\pp(\zero)=\m$.
\end{rem}

From now on, we discuss the relationship of radical thick subcategories of $\D$ with tame thick subcategories, which are defined below. 

\begin{dfn}
(1) A thick subcategory $\X$ of $\D$ is called {\it tame} if there is a subset $W$ of $\spec R$ such that $\X = \supp_{\D}^{-1}W$.
In this case, $W$ can be taken as a specialization-closed subset; in fact, it holds that $\supp_\D^{-1}W=\supp_\D^{-1}(\supp(\supp_\D^{-1}W))$.
Denote by $\Tame\D$ the set of tame thick subcategories of $\D$.\\
(2) We put $\X^\tame:=\supp_\D^{-1}(\supp \X)$ for each full subcategory $\X$ of $\D$.
This is the smallest tame thick subcategory of $\D$ containing $\X$.
In this sense, we call it the {\it tame closure} of $\X$. 
\end{dfn} 

For a topological space $X$, we denote by $\spcl X$ the set of specialization-closed subsets of $X$, that is, the unions of closed subsets of $X$.
The theorem below complements the bijection given in Theorem \ref{class}.

\begin{thm}\label{cmpr}
There is a diagram of maps of sets
$$
\xymatrix{
\rad \D \ar@<0.5ex>[r]^-{\spp} \ar@<0.5ex>[d]^{{()}^\tame}   &
\st\D\ar@<0.5ex>[l]^-{\spp^{-1}} \ar@<0.5ex>[d]^{\s^{-1}} \\
\Tame\D \ar@<0.5ex>[u]^{\inc}  \ar@<0.5ex>[r]^-{\supp} &  \spcl(\spec R), \ar@<0.5ex>[l]^-{\supp_{\D}^{-1}} \ar@<0.5ex>[u]^{\pp^{-1}}
}
$$
where $\inc$ stands for the inclusion map.
The horizontal maps are mutually inverse bijections, and the compositions of the maps ending at bottom sets are commutative (in particular, the composition of the two vertical maps starting from each bottom set is the identity).
Moreover, $\rad \D = \Tame \D$ if and only if the maps $\pp: \spc\D\rightleftarrows \spec R :\s$ are mutually inverse bijections.

\end{thm}

\begin{proof}
Let us prove the theorem step by step.

(1) We have already got the mutually inverse bijections $(\spp,\spp^{-1})$ in Theorem \ref{class}.

(2) Using Remark \ref{2}(1), we see that $W=\supp (\supp_{\D}^{-1} W)$ for a specialization-closed subset $W$ of $\spec R$.
This yields the mutually inverse bijections $(\supp,\supp_\D^{-1})$.

(3) It is easy to check that there is an equality $\supp_{\D}^{-1} W = \bigcap_{\p\in W^\complement} \s(\p)$ for any subset $W$ of $\spec R$.
This shows $\Tame\D\subseteq\rad\D$, and we have $()^\tame\cdot\inc=1$.

(4) It is seen from \eqref{3} that $\s^{-1}(\spp\X)=\supp\X$ for a full subcategory $\X$ of $\D$, which shows that $\s^{-1}$ is well-defined.
We claim that for a specialization-closed subset $W$ of $\spec R$ there are equalities
$$
\pp^{-1}(W)=\spp\{\K(\p)\mid\p\in W\}=\spp(\thick\{\K(\p)\mid\p\in W\}).
$$
Indeed, the second equality follows from Remark \ref{11}.
To show the first equality,  pick any $\P\in\pp^{-1}(W)$ and put $\p:=\pp(\P)\in W$.
Then $\p\in\II(\P)$, which implies $\K(\p)\notin\P$ and hence $\P\in\spp\K(\p)$.
Conversely, pick any $\p\in W$ and $\P\in\spp\K(\p)$.
Then $\K(\p)\notin\P$, which implies $\p\in\II(\P)$.
The unique maximality of $\pp(\P)$ shows that $\pp(\P)$ contains $\p$, and we get $\pp(\P)\in W$ since $W$ is specialization-closed.
Therefore $\P$ belongs to $\pp^{-1}(W)$.
Thus the claim follows, which shows that $\pp^{-1}$ is well-defined.
Since $\pp\cdot\s=1$ by Corollary \ref{maps2}(1), we see that $\s^{-1} \cdot \pp^{-1} = 1$.

(5) We have $\s^{-1}(\spp \X) = \supp \X = \supp(\X^\tame)$ for a full subcategory $\X$ of $\D$, which implies $\s^{-1} \cdot \spp = \supp \cdot ()^\tame$.
Using this equality and the fact that the horizontal maps are bijections which we have already seen in (1) and (2), we can easily deduce $()^\tame \cdot \spp^{-1} = \supp_{\D}^{-1} \cdot \s^{-1}$.

Finally, let us prove the last assertion of the theorem.
The `if' part is easily deduced from the commutative diagram.
To show the `only if' part, assume $\rad \D = \Tame \D$.
Then every prime thick subcategory of $\D$ is tame, and hence it is in the image of $\s$ by Corollary \ref{maps2}(2).
\end{proof}

Now we give a partial answer to Question \ref{5}.

\begin{cor}\label{6}
There are equalities and a one-to-one correspondence
$$
\xymatrix{
\th\dpf(R)=\Tame\dpf(R)=\rad\dpf(R)
 \ar@<0.5ex>[r]^-{\spp} &
\st\dpf(R)=\thom(\spc\dpf(R))\ar@<0.5ex>[l]^-{\spp^{-1}}.
}
$$
Furthermore, the maps in the diagram in Theorem \ref{cmpr} are all bijections for $\D=\dpf(R)$.
\end{cor}

\begin{proof}
The Hopkins--Neeman theorem \cite[Theorem 1.5]{Nee} says that all the thick subcategories of $\dpf(R)$ are tame.
It follows from Theorem \ref{cmpr} that every tame thick subcategory of $\dpf(R)$ is radical.
Hence we obtain the equalities $\th\dpf(R)=\Tame\dpf(R)=\rad\dpf(R)$.

Let us show the equality $\st\dpf(R)=\thom(\spc\dpf(R))$.
The homeomorphisms $\pp,\s$ induce mutually inverse bijections between the Thomason subsets of $\spc\dpf(R)$ and those of $\spec R$.
Note here that a subset of $\spec R$ is Thomason if and only if it is specialization-closed.
Therefore, for each $T\in\thom(\spc\dpf(R))$ we have $\pp(T)\in\thom(\spec R)=\spcl(\spec R)$, and $T=\pp^{-1}(\pp(T))$ belongs to $\st(\dpf(R))$ by Theorem \ref{cmpr}.
Conversely, for any $U\in\st(\dpf(R))$ we have $\s^{-1}(U)\in\spcl(\spec R)=\thom(\spec R)$ by Theorem \ref{cmpr} again, whence $U=\s(\s^{-1}(U))\in\thom(\spc\dpf(R))$.
We now conclude that $\st\dpf(R)=\thom(\spc\dpf(R))$.
Combining Theorem \ref{class}, we are done.
\end{proof}

\begin{rem}
Recall that the derived category $\dpf(R)$ has the structure of a tensor triangulated category, so that Balmer's theory can be applied to it.
The one-to-one correspondence in Corollary \ref{6} is identified with Theorem \ref{0Balcls} for $\T=\dpf(R)$.
Indeed, there are equalities
$$
\spc\dpf(R)=\bspc\dpf(R),\qquad
\spp=\bspp,
$$
where $\bspp$ stands for the Balmer support.

Let us show these equalities.
The latter one follows from the former one and the definitions of $\spp$ and $\bspp$.
The former equality follows from (i) and (ii) below; let $\P$ be a thick subcategory of $\dpf(R)$.

(i) Suppose that $\P$ belongs to $\spc\dpf(R)$.
Then $\P\ne\dpf(R)$.
As any thick subcategory of $\dpf(R)$ is a tensor ideal, so is $\P$.
Setting $\p:=\pp(\P)$, we have $\P=\s(\p)$ by Theorem \ref{bij}(1).
Let $X,Y\in\dpf(R)$ be objects such that $X\otimes_RY\in\P$.
Then $X_\p\otimes_{R_\p}Y_\p\cong0$ in $\dpf(R_\p)$, which implies either $X_\p\cong0$ or $Y_\p\cong0$.
Hence either $X\in\P$ or $Y\in\P$ holds, and we conclude that $\P\in\bspc\dpf(R)$.

(ii) Suppose that $\P\in\bspc\dpf(R)$, and let $I,J$ be maximal elements of $\II(\P)$.
Then neither $\K(I)$ nor $\K(J)$ is in $\P$.
Since $\P\in\bspc\dpf(R)$, we have $\K(I+J)=\K(I)\otimes_R\K(J)\notin\P$.
Hence $I+J\in\II(\P)$, and the maximality of $I,J$ shows $I=I+J=J$.
Thus $\II(\P)$ has a unique maximal element, and $\P\in\spc\dpf(R)$.
\end{rem}

We close this section by proving a theorem on the dimension of $\D$, which improves Theorem \ref{bij}(3).

\begin{thm}\label{dim}
\begin{enumerate}[\rm(1)]
\item
There is an inequality $\dim \D \ge \dim R$.
\item
Consider the following conditions:\quad \\
\quad
{\rm(a)} $\dim \D = \dim R$,\quad
{\rm(b)} $\dim \spc \D = \dim R$,\quad
{\rm(c)} $\spc \D\cong\spec R$,\quad
{\rm(d)} $\D = \dpf(R)$.\\
Then the implications {\rm(d)} $\Rightarrow$ {\rm(c)} $\Rightarrow$ {\rm(b)} $\Rightarrow$ {\rm(a)} hold.
If $R$ is a complete intersection local ring with positive Krull dimension, then the implication {\rm(a)} $\Rightarrow$ {\rm(d)} holds as well.
\end{enumerate}
\end{thm}

\begin{proof}
(1) The assertion follows from Corollary \ref{maps2}(1).

(2) Theorem \ref{bij}(1) shows the implication (d) $\Rightarrow$ (c), while (b) $\Rightarrow$ (a) is shown by Remark \ref{r} and (1).

To show (c) $\Rightarrow$ (b), let $f: \spc \D \to\spec R$ be a homeomorphism.
Then $\spc \D$ is a sober space since so is $\spec R$, and we obtain $\dim \spc \D = \dim \D$ by Remark \ref{r}.
Note that $f(\overline{\{\P\}}) = \overline{\{f(\P)\}}$ for each $\P\in\spc\D$.
Using Proposition \ref{t0}, we see that $\P\subseteq\Q$ in $\spc\D$ if and only if $f(\P)\supseteq f(\Q)$.
It follows that $\dim \D = \dim R$, and consequently, $\dim\spc\D=\dim R$.

Finally, we show the implication (a) $\Rightarrow$ (d), assuming that $(R, \m)$ is a local complete intersection with $d:= \dim R>0$.
Take a chain $\p_0 \subsetneq \cdots  \subsetneq \p_d = \m$ of prime ideals of $R$.
Applying $\s$ gives rise to a chain
$$
\zero=\s(\m)=\s(\p_d)\subsetneq \s(\p_{d-1}) \subsetneq \cdots \subsetneq \s(\p_0)
$$
of prime thick subcategories of $\D$.
Since $\D$ has dimension $d$, we have $\height\s(\p_{d-1})=1$.
It follows from Corollary \ref{maps2}(2) and the fact $\pp(\zero)=\m$ that $\pp^{-1}(\p_{d-1})=\{\s(\p_{d-1})\}$.
Replacing $\p$ with $\p_{d-1}$ in the proof of Theorem \ref{bij}(3), we obtain $\D = \dpf(R)$.
\end{proof}

\section{Spectra of singularity categories}

In this section, we investigate the spectra of singularity categories of commutative noetherian rings.
The flow will go similar to the previous section, while the results which will be obtained in this section are completely independent.

Throughout this section, we assume that $R$ is a commutative noetherian ring.
Recall that the {\em singularity category} $\ds(R)$ of $R$ is defined as the Verdier quotient
$$
\ds(R) := \db(R) / \dpf(R).
$$
By definition $\ds(R)$ is a triangulated category, and we can define its spectrum $\spc(\ds(R))$.
The purpose of this section is to explore the structure of the topological space $\spc(\ds(R))$.

We denote by $\sing R$ the {\em singular locus} of $R$, that is, the set of prime ideals $\p$ of $R$ such that the local ring $R_\p$ is singular (i.e. nonregular).
We equip $\sing R$ with the relative topology to regard it as a subspace of $\spec R$.
First of all, let us recall the definition of a {\em singular support}, which plays an important role in this section.

\begin{dfn}
For an object $M \in \ds(R)$, we define the {\em singular support} of $M$ by
$$
\ssupp M:= \{\p \in \sing R \mid M_\p \not\cong 0 \mbox{ in } \ds(R_\p) \}=\{\p\in\sing R\mid\pd_{R_\p}M_\p<\infty\}.
$$
For a full subcategory $\X \subseteq \ds(R)$, we define the {\em singular support} of $\X$ by
$$
\ssupp\X := \bigcup_{M \in \X} \ssupp M.
$$
\end{dfn}

Here we state some basic properties of singular supports, which correspond to Remark \ref{2}.

\begin{rem}\label{9}
\begin{enumerate}[\rm(1)]
\item
For $\p\in\sing R$ it holds that $\ssupp(R/\p)=\V(\p)$.
\item
For $M\in\ds(R)$ the set $\ssupp M$ is Zariski-closed.
\item 
For $M\in\ds(R)$ one has $\ssupp M = \emptyset$ if and only if $M \cong 0$.
\item
For $M\in\ds(R)$ one has $\ssupp(M[1])=\ssupp M$.
\item 
For $M,N\in\ds(R)$ the equality $\ssupp(M \oplus N) = \ssupp M \cup \ssupp N$ holds.
\item 
For an exact triangle $L \to M \to N \to L[1]$ in $\ds(R)$ it holds that $\ssupp M \subseteq \ssupp L \cup \ssupp N$.
\end{enumerate}
\noindent
In particular, for each subset $W$ of $\sing R$ the full subcategory
$$
\ssupp^{-1}W := \{M \in \ds(R) \mid \ssupp M \subseteq W \}.
$$
of $\ds(R)$ is a thick subcategory.
\end{rem}

We adopt the following definition of prime thick subcategories of $\ds(R)$.
We should compare this with Definition \ref{7}.
The way to define the set $\II(\X)$ is quite different, but the definition of a prime thick subcategory is similarly done.

\begin{dfn}
For each $\X \in \th(\ds(R))$ we define the set $\II(\X)$ of ideals of $R$ by
$$
\II(\X):=\{I \subseteq R \mid \V(I) \subseteq \sing R,\, R/I \not\in \X\}.
$$
We say that a thick subcategory $\P$ is {\it prime} if $\II(\P)$ has a unique maximal element.
Denote by $\pp(\P)$ the maximal element of $\II(\P)$ and by $\spc \ds(R)$ the set of prime thick subcategories of $\ds(R)$.
We equip $\spc \ds(R)$ with the topology defined in Definition \ref{top}.
\end{dfn}

Similarly as in Section 3, for each $\p \in \sing R$ we define the full subcategory $\s(\p)$ of $\ds(R)$ by
$$
\s(\p):= \{M \in \ds(R) \mid M_\p \cong 0 \mbox{ in } \ds(R_\p) \}.
$$
Analogous statements to Lemmas \ref{sprm2} and \ref{prm} hold true, although the proof is rather different.

\begin{lem}\label{sprm}
\begin{enumerate}[\rm(1)]
\item
Let $\p\in\sing R$.
Then $\s(\p)$ is a prime thick subcategory of $\ds(R)$ with $\pp(\s(\p))=\p$.
\item
For a thick subcategory $\X$ every maximal element of $\II(\X)$ belongs to $\sing R$.
In particular, $\pp(\P)$ belongs to $\sing R$ for any prime thick subcategory $\P$ of $\ds(R)$.
\end{enumerate}
\noindent
Thus one obtains a pair of maps
$$
\pp: \spc \ds(R) \rightleftarrows \sing R :\s.
$$
\end{lem}

\begin{proof}
(1) We easily verify the following, which shows that the subcategory $\s(\p)$ of $\ds(R)$ is thick.
\begin{equation}\label{8}
\s(\p)= \ssupp^{-1} \{\q \in \sing R \mid \q \nsubseteq\p \}.
\end{equation}
Since $\p$ is in $\sing R$, we have $(R/\p)_\p \not\cong 0$ in $\ds(R_\p)$.
Hence $\p$ belongs to $\II(\s(\p))$.
Any ideal $I$ belonging to $\II(\s(\p))$ satisfies $R/I \not\in \s(\p)$, which implies $I \subseteq \p$.
Thus, $\p$ is a unique maximal element of $\II(\s(\p))$.

(2) Pick a maximal element $\p$ of $\II(\X)$; note $\p\ne R$.
Let $a, b \in R\setminus\p$ elements with $ab \in \p$.
Then the ideals $\p+(a)$ and $\p : a$ strictly contain $\p$.
Therefore $R/\p+(a)$ and $R/ (\p: a) \cong\p+(a)/ \p$ belong to $\X$ by the maximality of $\p$.
The exact sequence $0 \to \p+(a)/ \p \to R/\p \to R/\p+(a) \to 0$ implies $R/\p \in \X$, which is a contradiction.
Thus $\p$ is prime.
As $\p\in\II(\P)$, we have $\V(\p)\subseteq\sing R$, which implies $\p\in\sing R$.
\end{proof}

The following proposition should be compared with Corollary \ref{maps2}.
They are similar, but there are some differences.
In particular, the inclusion relations between $\P$ and $\s(\pp(\P))$ are opposite.

\begin{prop}\label{maps}
\begin{enumerate}[\rm(1)]
\item
The maps $\s$ and $\pp$ are order-reversing with $\pp \cdot \s = 1$, and $\s$ is continuous.
\item
The map $\pp$ is also continuous, if either $\pp$ is injective and $\sing R$ is closed, or $\sing R$ is finite.
\item
Suppose that $R$ is a Gorenstein local ring.
For any $\P\in\spc\ds(R)$ there is an inclusion $\s(\pp(\P))\subseteq \P$.
\end{enumerate}
\end{prop}

\begin{proof}
(1) It is straightforward that $\s,\pp$ are order-reversing maps, while Lemma \ref{sprm}(1) implies $\pp \cdot \s = 1$.

(2) First of all, we claim that if $\sing R$ is closed and so is $\pp^{-1}(\V(\p))$ for each $\p\in\sing R$, then $\pp$ is continuous.
Indeed, write $\sing R=\V(I)$ with an ideal $I$ of $R$.
Each closed subset of $\sing R$ is of the form $\V(J)\cap\sing R=\V(I+J)$, where $J$ is an ideal of $R$.
Let $\p_1,\dots,\p_n$ be the minimal primes of $I+J$.
Then for each $i$ the prime ideal $\p_i$ is in the singular locus of $R$, and $\pp^{-1}(\V(\p_i))$ is closed by assumption.
Hence $\pp^{-1}(\V(J)\cap\sing R)=\bigcup_{i=1}^n\pp^{-1}(\V(\p_i))$ is also closed, and the claim follows.

Now, assume that $\pp$ is injective.
Then $\pp$ is a bijection with $\pp^{-1}=\s$.
Fix $\P\in\spc\ds(R)$.
We have
$$
\ssupp \P = \ssupp \s(\pp(\P)) = \{\q \in \sing R \mid \q \not\subseteq \pp(\P) \}
$$
by \eqref{8} and Remark \ref{9}(1).
Take any $\p\in\sing R$.
Then $\P\in\pp^{-1}(\V(\p))$ if and only if $\p\subseteq\pp(\P)$.
Also, $\P\in\spp(R/\p)$ if and only if $R/\p\notin\P$, and in this case $\p\in\II(\P)$ and hence $\p\subseteq\pp(\P)$.
If $\p\subseteq\pp(\P)$ and $R/\p\in\P$, then $\p\in\V(\p)=\ssupp(R/\p)\subseteq\ssupp\P$, which gives a contradiction.
Consequently, we obtain $\pp^{-1}(\V(\p)) = \spp (R/\p)$ for all $\p\in\sing R$.
The above claim shows that $\pp$ is continuous.

Next, assume that $\sing R$ is a finite set.
Then, in particular, $\sing R$ is a closed subset of $\spec R$, and by the above claim it suffices to show that for each $\p\in\sing R$ one has
$$
\textstyle
\pp^{-1}(\V(\p)) = \bigcup_{\q \in \V(\p)} \spp(R/\q)
$$
because by assumption the union is finite and hence it is closed.
Note that each $\q\in\V(\p)$ is in the singular locus of $R$.
A prime thick subcategory $\P$ of $\ds(R)$ belongs to the right-hand side of the above equality if and only if $R/\q\notin\P$ for some $\q\in\V(\p)$, if and only if there exists a prime ideal $\q$ of $R$ such that $\p\subseteq\q\in\II(\P)$, if and only if $\p\subseteq\pp(\P)$, if and only if $\P$ belongs to the left-hand side of the above equality.

(3) Consider the thick subcategory $\X:=\thick \{R/\p \mid \p\in\sing R,\,\p \not\subseteq \pp(\P)\}$ of $\ds(R)$.
We observe that $\X$ is contained in $\P$ and that the equality $\ssupp \X = \{\p \in \sing R \mid \p \not\subseteq \pp(\P) \}$ holds.
Therefore, $R/ \p$ belongs to $\X$ for every $\p \in \ssupp \X$.
Since $R$ is a Gorenstein local ring, we can apply \cite[Corollary 4.11]{Tak10} to get the equality $\X =\ssupp^{-1} (\ssupp \X)$, whose right-hand side coincides with $\s(\pp(\P))$ by \eqref{8}.
\end{proof}

Let $R$ be a local ring with residue field $k$.
For $M \in \db(R)$, we define its {\em complexity} by
$$
\cx_R M := \inf \{ c \in \ZZ_{\ge0}\mid \dim_k \Ext_R^n(M, k) \le r n^{c-1} \mbox{ for some $r\in\RR$ and for all } n \gg 0 \}.
$$
We refer the reader to \cite[\S4.2]{Av} for fundamental properties of this numerical invariant.
Since $\cx_R P = 0$ for $P \in \dpf(R)$, the complexity is well-defined on the isomorphism class of each object of $\ds(R)$.
Now we state and prove the following theorem, which should be compared with Theorem \ref{bij}.

\begin{thm}\label{sgbij}
Let $R$ be a local ring with closed singular locus and admitting a complex of finite positive complexity (e.g., let $R$ be a singular excellent local complete intersection).
The following are equivalent.
\begin{enumerate}[\rm(1)]
\item 
The maps $\pp: \spc \ds(R) \rightleftarrows \sing R :\s$ are mutually inverse bijections.
\item
The maps $\pp: \spc \ds(R) \rightleftarrows \sing R :\s$ are mutually inverse homeomorphisms.
\item 
The ring $R$ is a hypersurface.
\end{enumerate}
\end{thm}

\begin{proof}
The implication (2) $\Rightarrow$ (1) is clear.

We show the implication (1) $\Rightarrow$ (3).
Let $\C$ be the full subcategory of $\ds(R)$ consisting of complexes with finite complexity.
Then $\C$ is a thick subcategory, and $\C\ne\zero$ by assumption. 
If $\C$ does not contain the residue field $k$ of $R$, then $\II(\C)$ contains the maximal ideal $\m$ of $R$, that is, $\C$ is a prime thick subcategory with $\pp(\C) = \m$.
Hence $\C=\s(\pp(\C))=\s(\m) = \zero$, which gives a contradiction.
Thus $k$ is in $\C$, and $R$ is a complete intersection by \cite[Theorem 2.3]{Gul}.
There is a finitely generated $R$-module $M$ with complexity $1$ by \cite[Proposition 2.2]{Ber}.
The full subcategory $\C'$ consisting of objects of $\ds(R)$ with complexity at most $1$ is thick and nonzero.
An analogous argument as above shows that $k$ belongs to $\C'$.
Therefore, $R$ is a hypersurface.

Finally, we prove the implication (3) $\Rightarrow$ (2).
In view of (1) and (2) of Proposition \ref{maps}, it suffices to show that $\s$ is surjective.
We first claim that
$$
\s(\p) =\{X \in \ds(R) \mid R/\p \not\in \thick X \}
$$
for each $\p\in\sing R$.
Indeed, for each  $X\in\ds(R)$, one has $X \not\in \s(\p)$ if and only if $X_\p \not\cong 0$, if and only if $\V(\p)\subseteq \ssupp X$, if and only if $\ssupp(R/\p)\subseteq \ssupp X$.
Since $R$ is a local hypersurface, the last condition is equivalent to saying that $R/\p \in \thick X$ by \cite[Theorem 6.8]{Tak10}.
Thus the claim follows.

Let $\P$ be a prime thick subcategory of $\ds(R)$.
Setting $\p = \pp(\P)$, we have $\s(\p)=\s(\pp(\P))\subseteq\P$ by Proposition \ref{maps}(3).
Take any object $X\in\P$.
Then $\thick X$ is contained in $\P$ and $R/\p=R/\pp(\P)$ does not belong to $\P$, which yields $R/\p\notin\thick X$.
The above claim implies that $X$ belongs to $\s(\p)$, and we obtain $\s(\p)=\P$.
It follows that $\s$ is surjective, which completes the proof of the assertion.
\end{proof}

\begin{rem}
There is a case where $R$ is not a hypersurface but $\pp,\s$ give mutually inverse homeomorphisms.
For instance, let $R$ be a Cohen--Macaulay local ring with quasi-decomposable maximal ideal (in the sense of \cite{fiber}) which is locally a hypersurface on the punctured spectrum.
Then, by \cite[Theorem 4.5]{fiber} every thick subcategory of $\ds(R)$ is of the form $\ssupp^{-1} W$ with $W$ a specialization-closed subset of $\sing R$, even if $R$ is not a hypersurface.
Hence, the same argument as above proves what we want.
For such a ring $R$, the existence of complexes of finite positive complexity fails.
\end{rem}

Next, as we did in the previous section, we introduce tame thick subcategories of $\ds(R)$ and relate them with radical ones.
Most of the arguments in the previous section does work for $\ds(R)$ just by replacing $\D,\supp_\D,\spec,\K(\p)$ with $\ds(R),\ssupp,\sing,R/\p$ respectively.
We will give definitions, properties, results, and proofs that are essentially different from the ones given in the previous section.

\begin{dfn}
A thick subcategory $\X$ of $\ds(R)$ is said to be {\it tame} if there exists a subset $W$ of $\sing R$ such that $\X = \ssupp^{-1}W$.
Denote by $\Tame\ds(R)$ the set of tame thick subcategories of $\ds(R)$.
We put $\X^\tame:=\ssupp^{-1}(\ssupp \X)$ for each full subcategory $\X$ of $\ds(R)$.
Note that $\X^\tame$ is the smallest tame thick subcategory of $\ds(R)$ containing $\X$, and we call $\X^\tame$ the {\em tame closure} of $\X$.
\end{dfn} 

\begin{prop}\label{12}
\begin{enumerate}[\rm(1)]
\item
For $W\subseteq\sing R$ one has $\ssupp^{-1}(\ssupp(\ssupp^{-1}W))=\ssupp^{-1}W=\bigcap_{\p\in W^\complement}\s(\p)$.
\item
For a full subcategory $\X$ of $\ds(R)$ it holds that $\s^{-1}(\spp\X)=\ssupp\X=\ssupp(\X^\tame)$.
\item
For a specialization-closed subset $W$ of $\sing R$ the equality $\pp^{-1}(W)=\spp(\thick\{R/\p\mid\p\in W\})$ holds.
If $R$ is a Gorenstein local ring, then the equality $\pp^{-1}(W)=\spp(\ssupp^{-1}W)$ also holds.
\end{enumerate}
\end{prop}

\begin{proof}
The only statement essentially different from what we got in the previous section is the latter assertion of (3), so we only give a proof of it.
Let $\P$ be a prime thick subcategory of $\ds(R)$, and set $\p:=\pp(\P)$.
If $\P\in\pp^{-1}(W)$, then $\p\in W$ and $\ssupp(R/\p)=\V(\p)\subseteq W$ as $W$ is specialization-closed.
Hence $R/\p\in\ssupp^{-1}W$.
We have $R/\p\notin\P$ since $\p\in\II(\P)$.
Thus $\P\in\spp(R/\p)\subseteq\spp(\ssupp^{-1}W)$.
Conversely, suppose $\P\in\spp(\ssupp^{-1}W)$.
Then $\P\in\spp X$ for some $X\in\ssupp^{-1}W$.
If $X_\p\cong0$, then by Proposition \ref{maps}(3) we get $X\in\s(\p)=\s(\pp(\P))\subseteq\P$, which implies $\P\notin\spp X$, a contradiction.
Therefore $X_\p\not\cong0$, which gives $\p\in\ssupp X\subseteq W$.
Thus $\pp(\P)=\p\in W$, and we obtain $\P\in\pp^{-1}(W)$.
Now we conclude that the equality $\pp^{-1}(W)=\spp(\ssupp^{-1}W)$ holds.
\end{proof}

We should compare the following Theorem \ref{10} with Theorem \ref{cmpr}.

\begin{thm}\label{10}
Suppose that $R$ is a Gorenstein local ring.
Then there is a diagram of maps of sets
$$
\xymatrix{
\rad \ds(R) \ar@<0.5ex>[r]^-{\spp} \ar@<0.5ex>[d]^{{()}^\tame}   &
\st\ds(R)\ar@<0.5ex>[l]^-{\spp^{-1}} \ar@<0.5ex>[d]^{\s^{-1}} \\
\Tame\ds(R) \ar@<0.5ex>[u]^{\inc}  \ar@<0.5ex>[r]^-{\ssupp} &  \spcl(\sing R) \ar@<0.5ex>[l]^-{\ssupp^{-1}} \ar@<0.5ex>[u]^{\pp^{-1}}
}
$$
The horizontal maps are mutually inverse bijections, and the compositions of the maps ending at bottom sets are commutative.
Furthermore, $\rad \ds(R) = \Tame \ds(R)$ if and only if the maps $\pp: \spc\ds(R)\rightleftarrows \sing R :\s$ are mutually inverse bijections.
\end{thm}

\begin{proof}
The only statement essentially different from Theorem \ref{cmpr} is the new commutativity relations $\spp\cdot\inc=\pp^{-1}\cdot\ssupp$ and $\inc\cdot\ssupp^{-1}=\spp^{-1}\cdot\pp^{-1}$, but the former follows from the latter and the two horizontal one-to-one correspondences in the diagram.
Take any $W\in\spcl(\sing R)$.
Then we have
$$
(\inc\cdot\ssupp^{-1})(W)=\ssupp^{-1}W=(\spp^{-1}\cdot\spp)(\ssupp^{-1}W)=(\spp^{-1}\cdot\pp^{-1})(W),
$$
where the last equality follows from Proposition \ref{12}(3).
We conclude that $\inc\cdot\ssupp^{-1}=\spp^{-1}\cdot\pp^{-1}$.
\end{proof}

The following result corresponds to Corollary \ref{6}, which also gives a partial answer to Question \ref{5}.

\begin{cor}
Let $R$ be a local hypersurface with closed singular locus.
One then has the following equalities and one-to-one correspondence, and the maps in the diagram in Theorem \ref{10} are all bijections.
$$
\xymatrix{
\th\ds(R)=\Tame\ds(R)=\rad\ds(R) \ar@<0.5ex>[r]^-{\spp} &
\st\ds(R)=\thom(\spc\ds(R))\ar@<0.5ex>[l]^-{\spp^{-1}}.
}
$$
\end{cor}

\begin{proof}
It follows from \cite[Main Theorem]{Tak10} that every thick subcategory of $\ds(R)$ is tame.
It is easy to see that a subset of $\sing R$ is Thomason if and only if it is specialization-closed.
Now essentially the same argument as in the proof of Corollary \ref{6} works.
\end{proof}

Next, we further study prime thick subcategories of $\ds(R)$.
We introduce the following notion.

\begin{dfn}
For a thick subcategory $\X$ of $\ds(R)$, we define the {\em tame interior} $\X_\tame$ of $\X$ as the largest tame thick subcategory of $\ds(R)$ contained in $\X$.
\end{dfn}

In the case where $R$ is a Gorenstein local ring, one can describe tame interiors explicitly as follows.

\begin{lem}\label{mxtm}
Let $R$ be a Gorenstein local ring.
For a thick subcategory $\X$ of $\ds(R)$ one has
$$
\X_\tame= \thick \{R/\p \mid \p \in \sing R\text{ and }\V(\p) \cap \max \II(\X) = \emptyset \}.
$$
\end{lem}

\begin{proof}
Let $\Y$ be the right-hand side.
The set $W:=\{\p \in \sing R \mid \V(\p) \cap \max \II(\X) = \emptyset\}$ is a specialization-closed subset of $\sing R$, and $\Y$ is a tame thick subcategory with singular support $W$ by \cite[Corollary 4.11]{Tak10}.
If $\p\in\sing R$ is such that $R/\p\notin\X$, then $\p\in\II(\X)$ and $\V(\p) \cap \max \II(\X) \neq \emptyset$.
Therefore, $\ssupp R/\p=\V(\p)$ is not contained in $W$, which shows that $R/\p$ is not in $\ssupp^{-1}W=\Y$.
Thus, $\Y$ is contained in $\X$.

It remains to show the maximality of $\Y$, and for this, it suffices to verify that if $\p\in\sing R$ is such that $\ssupp^{-1}\V(\p)$ is contained in $\X$, then $\ssupp^{-1}\V(\p)$ is contained in $\Y$.
Assume that $\ssupp^{-1}\V(\p)$ is not contained in $\Y$.
Then $\V(\p)$ is not contained in $W$, and we find an element $\q \in \V(\p)$ such that $\V(\q) \cap \max \II(\X) \neq \emptyset$.
Hence $\p$ is contained in some $\r\in\max \II(\X)$.
Then $R/ \r \in \ssupp^{-1} \V(\p)$ but $R/\r \not\in \X$, which is a contradiction.
Consequently, $\ssupp^{-1}\V(\p)$ is contained in $\Y$, and we are done.
\end{proof}

We obtain a characterization of the prime thick subcategories of $\ds(R)$ in terms of tame interiors.

\begin{prop}\label{chprm}
Let $R$ be a Gorenstein local ring.
A thick subcategory $\X$ of $\ds(R)$ is prime if and only if $\X_\tame = \s(\p)$ for some $\p \in \sing R$.
When this is the case, one has $\pp(\X) = \p$.
\end{prop}

\begin{proof}
Suppose that $\X$ is a prime thick subcategory of $\ds(R)$.
Set $\p=\pp(\X)$.
Proposition \ref{maps}(3) implies that $\s(\p)$ is contained in $\X$, while $\s(\p)$ is tame by \eqref{8}.
Hence $\s(\p)$ is contained in $\X_\tame$.
Let $\q\in\sing R$ be such that $\V(\q)\cap\max\II(\X)=\emptyset$.
Then $\q$ is not contained in $\p$, and $R/\q$ is in $\s(\p)$.
Using Lemma \ref{mxtm}, we obtain $\X_\tame=\s(\p)$.

Conversely, assume $\X_\tame=\s(\p)$.
Then the equality $(\ssupp(\X_\tame))^\complement=(\ssupp\s(\p))^\complement$ holds, which gives
$$
\{\q \in \sing R \mid \V(\q) \cap \max \II(\X) \neq \emptyset \} = \{\q \in \sing R \mid \q \subseteq \p \}
$$
by Lemma \ref{mxtm} and \eqref{8}.
The right-hand side has a unique maximal element, which is $\p$.
Note that every maximal element of $\II(\X)$ is also maximal in the left-hand side.
Thus $\p$ is the only maximal element of $\II(\X)$, which means that $\X$ is a prime thick subcategory with $\pp(\X) = \p$.
\end{proof}

Except trivial examples $\s(\p)$, it is difficult in general to find prime thick subcategories of $\ds(R)$.
Our next aim is to provide methods to construct new prime thick subcategories from given or trivial ones.
 
\begin{prop}\label{med}
Let $R$ be a Gorenstein local ring.
Let $\P$, $\Q$ be prime thick subcategories of $\ds(R)$.
\begin{enumerate}[\rm(1)]
\item
Assume $\P \subseteq \Q$ and $\pp(\P) = \pp(\Q)$.
Then any thick subcategory $\X$ with $\P \subseteq \X \subseteq \Q$ is a prime thick subcategory of $\ds(R)$ satisfying $\pp(\X) = \pp(\Q)$.
\item
If $\pp(\P) \subseteq \pp(\Q)$, then $\P \cap \Q$ is a prime thick subcategory of $\ds(R)$ satisfying $\pp(\P \cap \Q) = \pp(\Q)$. 
\end{enumerate}
\end{prop}

\begin{proof}
(1) There are inclusions $\P_\tame \subseteq \X_\tame \subseteq \X \subseteq \Q$.
By Proposition \ref{chprm} we have $\P_\tame = \s(\pp(\P)) = \s(\pp(\Q)) = \Q_\tame$, which is the largest tame thick subcategory contained in $\Q$.
Therefore, $\X_\tame = \P_\tame = \s(\pp(\P))$.
Applying Proposition \ref{chprm} again, we see that $\X$ is prime and $\pp(\X) = \pp(\Q)$.

(2) Proposition \ref{chprm} implies $\Q_\tame = \s(\pp(\Q)) \subseteq \s(\pp(\P))=\P_\tame$.
As $\Q_\tame \subseteq \Q$ and $\Q_\tame \subseteq \P_\tame \subseteq \P$, we get $\Q_\tame \subseteq \P \cap \Q \subseteq \Q$.
It follows from (1) that $\P \cap \Q$ is prime and $\pp(\P \cap \Q) = \pp(\Q)$.
\end{proof}

We introduce the notion of covers to state our next results.

\begin{dfn}
Let $\T$ be a triangulated category.
Let $\X,\Y$ be thick subcategories of $\T$.
Then we say that $\Y$ is a {\it cover} of $\X$ if $\Y$ properly contains $\X$ and there is no thick subcategory $\cZ$ of $\T$ with $\X \subsetneq \cZ \subsetneq \Y$. 
\end{dfn}

It is unclear in general whether a cover of a given thick subcategory exists or not.
The following proposition gives us sufficient conditions for the existence of covers.

\begin{prop}\label{a}
Let $R$ be either
\begin{itemize}
\item
a complete intersection which is a quotient of a regular local ring, or
\item
a Cohen--Macaulay local ring with quasi-decomposable maximal ideal which is locally a hypersurface on the punctured spectrum.
\end{itemize}
Then the zero subcategory $\zero$ of $\ds(R)$ admits a cover.
If in addition $\sing R$ is finite, then $\s(\p)$ admits a cover as well for each $\p\in\sing R$.
\end{prop}

\begin{proof}
In either case, there is an order isomorphism between $\th(\ds(R))$ and the set of specialization-closed subsets of a noetherian topological space $X$; see \cite{fiber, Ste}.
By \cite[Lemma 2.9]{Mat}, it restricts to an order isomorphism between the set $T:=\{\thick M \mid M \in \ds(R)\}$ and the set of closed subsets of $X$.
Thus $T$ satisfies the descending chain condition (with respect to the inclusion relation), and we can take a minimal element $\X := \thick M$ of $T\setminus\{\zero\}$.
We can easily check that $\X$ is minimal among all nonzero thick subcategories, i.e., $\X$ is a cover of $\zero$.
The last assertion of the proposition is shown similarly.
\end{proof}

We obtain a sufficient condition for a given thick subcategory to be prime, using the notion of covers.

\begin{prop}\label{b}
Let $R$ be a Gorenstein local ring.
Let $\X$ be a non-tame thick subcategory of $\ds(R)$ and $\p \in \sing R$.
If $\X$ is a cover of $\s(\p)$, then $\X$ is a prime thick subcategory of $\ds(R)$ with $\pp(\X) = \p$. 
\end{prop}

\begin{proof}
By \eqref{8}, the thick subcategory $\s(\p)$ is tame.
As $\s(\p)$ is contained in $\X$ and $\X$ is not tame, we have $\s(\p) \subseteq \X_\tame \subsetneq \X$.
The assumption that $\X$ is a cover of $\s(\p)$ implies $\s(\p) = \X_\tame$.
It follows from Proposition \ref{chprm} that $\X$ is a prime thick subcategory of $\ds(R)$ such that $\pp(\X) = \p$.
\end{proof}

As another application of covers, we get criteria for a Cohen--Macaulay local ring to be a hypersurface.

\begin{thm}\label{c}
Let $R$ be a singular Cohen--Macaulay local ring possessing a complex of finite positive complexity. 
Assume that $\zero$ admits at least one cover (e.g., $R$ is a complete intersection which is a quotient of a regular local ring).
Then the following are equivalent.\\
{\rm(1)} $R$ is a hypersurface.\qquad
{\rm(2)} Every cover of $\zero$ is tame.\qquad
{\rm(3)} There is a tame cover of $\zero$.
\end{thm}

\begin{proof}
The implication $(1) \Rightarrow (2)$ follows from \cite[Theorem 6.8]{Tak10}, while the implication $(2) \Rightarrow (3)$ is trivial.
Let us show the implication $(3)\Rightarrow(1)$.
We denote by $\m$ the maximal ideal of $R$, by $k=R/\m$ the residue field of $R$, and put $d=\dim R$.
Let $\X$ be a tame  cover of $\zero$.
Then $\X = \ssupp^{-1}\{\m\}$ since it is minimal among the nonzero tame thick subcategories of $\ds(R)$.
Again the minimality of $\X$ shows that for each nonzero object $X\in\X$ one has $\thick X=\X=\ssupp^{-1}\{\m\}$, which contains $k$.

By assumption, there exists a maximal Cohen--Macaulay $R$-module $M$ with finite positive complexity.
Take a maximal regular sequence $\xx=x_1,\dots,x_d$ on $R$, and set $N=M/\xx M$.
Then $N$ is an $R$-module of finite length and with finite positive complexity.
As $N$ belongs to $\ssupp^{-1}(\m)$, the above argument says that $\thick N$ contains $k$.
Since the full subcategory of objects of $\ds(R)$ with finite complexity is thick, $k$ has finite complexity.
Therefore, $R$ is complete intersection by \cite[Theorem 2.3]{Gul} and $\cx_RC= \cx_R k = \codim R$ for any $0\neq C \in \ds(R)$.
It follows from \cite[Proposition 2.2]{Ber} that there is a finitely generated $R$-module $L$ with complexity $1$.
Then $\codim R = \cx_R L = 1$, which means that $R$ is a hypersurface.
\end{proof}

\begin{rem}
Let $(R,\m)$ be a Gorenstein singular local ring.
If $\pp^{-1}(\m)$ consists only of $\s(\m)=\zero$, then $\ssupp^{-1}\{\m\}$ is a unique cover of $\zero$.
Indeed, let $\X\ne\zero$ be a thick subcategory of $\ds(R)$.
If $\X_\tame = \zero$, then $\X \in \pp^{-1}(\m)$ by Proposition \ref{chprm}, and $\X= \zero$.
This contradiction shows $\X_\tame \neq \zero$.
As $\m\in\ssupp(\X_\tame)$,
$$
\zero\subsetneq\ssupp^{-1}\{\m\}\subseteq\ssupp^{-1}\ssupp(\X_\tame)=\X_\tame\subseteq\X.
$$
This proves that $\ssupp^{-1}\{\m\}$ is a unique cover of $\zero$.
In particular, every cover of $\zero$ is tame.
By Theorem \ref{c}, we obtain another proof of Theorem \ref{sgbij} in the Gorenstein case.
\end{rem}

Finally, we prove a result corresponding to Theorem \ref{dim}, which is an application of Theorem \ref{c}.

\begin{thm}\label{dimsg}
\begin{enumerate}[\rm(1)]
\item
One has the inequality $\dim \ds(R) \ge \dim \sing R$.
\item
Consider the following four conditions.\\
\begin{tabular}{ll}
\qquad
{\rm(a)} $\dim \ds(R) = \dim \sing R$.
&{\rm(b)} $\dim \spc \ds(R) = \dim \sing R$.\\
\qquad
{\rm(c)} $\spc \ds(R)\cong\sing R$.
&{\rm(d)} $R$ is a hypersurface.
\end{tabular}
\\
Then the implications {\rm(c)} $\Rightarrow$ {\rm(b)} $\Rightarrow$ {\rm(a)} holds.
The implication {\rm(d)} $\Rightarrow$ {\rm(c)} holds if $R$ is a singular local ring with closed singular locus.
The implication {\rm(a)} $\Rightarrow$ {\rm(d)} holds if $R$ is a Gorenstein local ring with closed singular locus possessing a complex of finite positive complexity.
\end{enumerate}
\end{thm}

\begin{proof}
(1) Proposition \ref{maps}(1) shows the assertion.

(2) It is clear that (c) implies (b), while it follows from (1) and Remark \ref{r} that (b) implies (a).
If $R$ is local and $\sing R\ne\emptyset$ is closed, then Theorem \ref{sgbij} shows that (d) implies (c).

Let us show that (a) implies (d) under the assumption that $(R,\m)$ is a Gorenstein local ring with closed singular locus possessing a complex of finite positive complexity.
Set $n:=\dim \ds(R) = \dim \sing R$.
Take a chain $\p_0 \subsetneq \cdots \subsetneq \p_n = \m$ in $\sing R$, and apply $\s$.
We get a chain
$$
\zero=\s(\m)=\s(\p_n)\subsetneq \s(\p_{n-1}) \subsetneq \cdots \subsetneq \s(\p_0)
$$
in $\spc\ds(R)$.
As $\dim\ds(R)=n$, we have $\height\s(\p_{n-1})=1$.
Fix a prime thick subcategory $\P$ of $\ds(R)$ with $\height \P = 1$.
Then $\P$ is a cover of $\s(\m)=\zero$.
If $\P$ is not tame, then $\P \in \pp^{-1}(\m)$ by Proposition \ref{b}.
Hence $\s(\p_{n-1}) \cap \P \in \pp^{-1}(\m)$ by Propositions \ref{med}(2) and \ref{maps}(1).
As $\s(\p_{n-1})$ is tame by \eqref{8}, it is not equal to $\P$.
Thus $\s(\p_{n-1})$ and $\P$ are distinct prime thick subcategories of height $1$, which forces us to have $\s(\p_{n-1}) \cap \P = \zero$.
Similarly to the last paragraph of the proof of Theorem \ref{bij}, for a maximal Cohen--Macaulay $R$-module $M$ in $\P$ and a maximal regular sequence $\xx$ on $R$, we have $M/\xx M \in \s(\p_{n-1}) \cap \P=\zero$, and $M \cong 0$ in $\ds(R)$, whence $\P = \zero$, a contradiction.  
Thus any height $1$ prime thick subcategory of $\ds(R)$ is tame, and so is every cover of $\zero$.
It follows from Theorem \ref{c} that $R$ is a hypersurface.
\end{proof}


\end{document}